\documentclass{article}[16pt]
\usepackage{amssymb}
\usepackage{amsthm}
\usepackage{amsmath}
\usepackage{tikz}
\usepackage{multicol}
\usepackage{enumitem}

\newtheorem{remark}{Remark}
\newtheorem{lemma}{Lemma}

\newtheorem*{problem}{Problem}
\newtheorem{corollary}{Corollary}
\newtheorem{theorem}{Theorem}
\newtheorem*{thm}{Main Theorem}

\begin{document}
\title{Characterization of the lengths of binary circular words containing no squares other than 00, 11, and 0101}
\author{James D. Currie\\
Department of Mathematics \&
Statistics\\
The University of Winnipeg\thanks{The author is
supported by an NSERC Discovery Grant}\\
{\tt currie@uwinnipeg.ca}\vspace{.1in}\\
Jesse T. Johnson\\
Department of Mathematics \&
Statistics\\
University of Victoria\\
{\tt jessejoho@gmail.com}}
\maketitle \abstract{\noindent We characterize exactly the lengths of binary circular words containing no squares other than 00, 11, and 0101. Key words: combinatorics on words, circular words, necklaces, square-free words, non-repetitive sequences}

\section{Introduction}
Combinatorics on words started with the work of Thue  \cite{thue06}, who showed the existence of arbitrarily long square-free words over a three letter alphabet. Thue studied circular words also \cite{thue12}, and completely characterized the circular overlap-free words on two letters. Circular words have been relatively unexplored until recently. In 2002, the first author characterized the lengths for which ternary circular square-free words exist: 

\begin{theorem}\label{ternary}
For every positive integer $n$ other than 5, 7, 9, 10, 14, or 17, there is a ternary circular square-free word of length $n$.
\end{theorem} 

Several other proofs \cite{shur10,clokie,mol} of this theorem have now been given, signaling increasing interest in circular words. The general question of when there is a circular word avoiding some pattern has also begun to receive attention \cite{fitz,mousavi}. 
Circular words avoiding patterns seem harder to understand than linear words; while the set of linear words avoiding some pattern is closed under taking factors, this is not true for the circular version. 

Square-free words are objects of continuing interest in combinatorics on words. As proven by Thue, there exist arbitrarily long words over $\{a,b,c\}$ that contain no square factors. On
the other hand, one quickly checks that every word over $\{0,1\}$ of length 4 or greater contains a square. A binary word with as few square factors as possible was found by Fraenkel and Simpson \cite{FS}. 
\begin{theorem}\label{FSHN} There exist arbitrarily long words over $\{0,1\}$ avoiding all factors of the form $xx$, $x\ne 0, 1, 01$. \end{theorem}

Simpler proofs of this result have been found \cite{harju,rampersadshallit}.
Call a binary word containing no squares other than 00, 11, and 0101 an {\bf FS word}. Harju and Nowotka have shown \cite{harju2} that there are arbitrarily long circular FS words.
It is natural to ask: For exactly which lengths are there circular FS words? We answer this question completely:

\begin{thm}\label{main} There is a circular FS word of length $m$ exactly when $m$ is a non-negative integer other than 9, 10, 11, 13, 15, 16, 17, 18, 21, 22, 23, 25, 26, 27, 29, 31, 32, 33, 34, 35, 37, 40, 41, 42, 45, 47, 49, 53, 56, 59, 61, 64, and 73.
\end{thm}

It would be interesting to probe the structure of circular FS words more deeply.

\begin{problem} For each positive integer $n$, how many circular FS words  are there of length $n$?
\end{problem}

\section{Preliminaries}

For general background on combinatorics on words, see the works of Lothaire   \cite{lothaire83,lothaire02}.
Let $\Sigma$ be a finite set. We refer to $\Sigma$ as an {\bf alphabet}, and its elements as {\bf letters}.  We denote by $\Sigma^*$ the free monoid over $\Sigma$, with identity $\epsilon$, the {\bf empty word}. We call the elements of $\Sigma^*$ {\bf words}. Informally, we think of the elements of $\Sigma^*$ as finite strings of letters, and of its binary operation as concatenation. Thus, if $u=u_1u_2\cdots u_n$, $u_i\in\Sigma$ and $v=v_1v_2\cdots v_m$, $v_j\in\Sigma$, then $uv=u_1u_2\cdots u_nv_1v_2\cdots v_m$. In this case, we say that $u$ is a {\bf prefix} of $uv$ and $v$ is a {\bf suffix}. More generally, if $w=uvz$, then $v$ is a {\bf factor} of $w$.   We say that $v$ appears in $w$ at {\bf index} $i$ in the case where $|u|=i-1$. 
We will work in particular with the alphabets  $B=\{0,1\}$, $S=\{a,b,c\}$,  and $T=\{a,b,c,d\}$. Words over $S$ are called {\bf ternary words}.

A word of the form $s=uu$, $u\ne \epsilon$ is called a {\bf square}. Thus a square $uu$ has period $|u|$. We write $z^2$ for $zz$. A word $w$ which doesn't contain a square factor is said to be {\bf square-free}.
We call a word over $B$ containing no square factors  other than 00, 11, and 0101 an {\bf FS word} (for Fraenkel-Simpson word).

 If $u=u_1u_2\cdots u_n$, $u_i\in\Sigma$, then the {\bf length} of $u$ is defined to be $n$, the number of letters in $u$, and we write $|u|=n$. The set of words of length $m$ over $\Sigma$ is denoted by $\Sigma^m$. We use $\Sigma^{\ge n}$ to denote the set of words over $\Sigma$ of length at least $n$. For $a\in\Sigma$, $u\in\Sigma^*$, we denote by $|u|_a$ the number of occurrences of $a$ in $u$.  
 
If $w=uv$, then define $wv^{-1}=u$. Thus $vu=v(uv)v^{-1}$, and we refer to $vu$ as a {\bf conjugate} of $uv$. The relation `$a$ is a conjugate of $b$' is an equivalence relation on $\Sigma^*$, and we refer to the equivalence classes of $\Sigma^*$ under this equivalence relation as {\bf circular words}. If $w\in\Sigma^*$, we denote the circular word containing $w$ by $[w]$.We may consider the indices $i$ of the letters of a circular word $[u]=[u_1u_2\cdots u_n]$ to belong to ${\mathbb Z}_n$, the integers modulo $n$. Thus $u_{n+1}=u_1$, for example. 
If $[w]$ is a circular word and $v\in\Sigma^*$, we say that $v$ is a {\bf factor} of $[w]$ if $v$ is a factor of an element of $[w]$, i.e., if $v$ is a factor of a conjugate of $w$. A circular word $[w]$ is {\bf square-free} if no factor of $[w]$ is a square.

A circular word $[w]$ is called an
{\bf FS circular word}, if every conjugate of $w$ is an FS word.

Let $\Sigma$ and $T$ be alphabets. A map $\mu:\Sigma^*\rightarrow T^*$ is called a {\bf morphism} if it is a monoid homomorphism, that is, if $\mu(uv)=\mu(u)\mu(v)$, for $u,v\in\Sigma^*$.
A morphism $f:\Sigma^*\rightarrow B^*$ such that $f(w)$ is an FS word whenever $w$ is square-free is called an {\bf FS morphism}.

A ternary $w$ such that  for any letters $x,y\in\{a,b,c\}$,
 $$|w|_x-1\le|w|_y\le|w_x|+1$$
 is called  {\bf level}. The authors recently proved the following:

\begin{theorem}\cite{curriejohnson,johnson}\label{ternary}
There is  a level ternary circular square-free word of length $n$,
exactly when $n$ is a positive integer, $n\ne 5, 7, 9, 10, 14, 17$. \end{theorem}

\section{Constructing circular FS words}

We begin by proving a generalization of an approach used by Harju and Nowotka \cite{harju}. They used  it to demonstrate that a particular morphism applied to  ternary square-free words gave FS words. Here, we use it to find FS morphisms  on alphabets of any size:
\begin{lemma}\label{HN} Fix $n\ge 3.$
Suppose $f:\Sigma_n^*\rightarrow\Sigma_2^*$ is a morphism satisfying these conditions:
\begin{enumerate}
\item\label{1} For any square-free $v\in\Sigma_n^3$, $f(v)$ is 
an FS word.
\item There is a word $p\in\Sigma_2^*$, $|p|\ge 3$, such that:
\begin{enumerate}
\item\label{2a} For each $a\in\Sigma_n$, $p$ is a prefix of $f(a)$.
\item\label{2b} If $a_i\in \Sigma_n$, $1\le i \le \ell$, and $f(a_1a_2\cdots a_\ell)=qpr$ for some words $q,r\in\Sigma_2^*$, then $q=\epsilon$ or $q=f(a_1a_2\cdots a_j)$, some $j\le\ell$.
\end{enumerate}
\end{enumerate}
Then $f$ is an FS morphism.
\end{lemma}
\begin{proof}
To begin with, note that the conditions imply that if $a,b\in\Sigma_n$ and $f(a)$ is a prefix of $f(b)$, then $a=b$. Otherwise, $aba$ is a square-free word of length 3, with square prefix $f(a)f(a)$. However, $|f(a)|\ge |p|\ge 3$, so $f(a)f(a)\ne 00, 11$, or 0101. This contradicts Condition~\ref{1}.
 
For the sake of getting a contradiction, consider a square-free word $w=w_1w_2\cdots w_m$, with the $w_i\in\Sigma_n$, such that $f(w_1w_2\cdots w_m)$ contains a square $xx$, $x\ne\epsilon,0,1,01$. Let $m$ be as small as possible. By Condition~\ref{1}, $m\ge 4$. Since $m$ is minimal, write $$xx=W_1^{\prime\prime}W_2\cdots W'_m,$$ 
\begin{eqnarray*}
\mbox{where }f(w_1)&=&W_1'W_1^{\prime\prime}, \/W_1^{\prime\prime}\ne \epsilon\\
f(w_i)&=&W_i, \/2\le i\le m-1\\
f(w_m)&=&W_m'W_m^{\prime\prime}, \/W_m'\ne \epsilon.
\end{eqnarray*}
As per Condition~\ref{2a}, write $W_2=pW_2''$.
\subsubsection*{Case A: $|x|<|W_1''|$ or $|x|<|W_m'|$} 
If $|x|<|W_1''|$, write $W_1''=xW_1^{\prime\prime\prime}$, $W^{\prime\prime\prime}_1\ne \epsilon$.   Then we find the second copy of $x$ in $xx$ can be written   
$$x=W_1^{\prime\prime\prime}W_2\cdots W_m'=W^{\prime\prime\prime}_1pW''_2\cdots W_m'.$$ However, then $$f(w_1)=W_1'W_1''=W_1'xW_1^{\prime\prime\prime}=W_1'W^{\prime\prime\prime}_1pW''_2\cdots W_m'W_1^{\prime\prime\prime}$$ contains an instance of $p$ at an index which contradicts Condition~\ref{2b}.

Similarly, if $|x|<|W_m'|$, write $W_m'=W_m^{\prime\prime\prime}x$, $W^{\prime\prime\prime}_m\ne \epsilon$.   Then we find the first copy of $x$ in $xx$ can be written   
$$x=W_1^{\prime\prime}W_2\cdots W_m^{\prime\prime\prime}=W^{\prime\prime\prime}_1pW''_2\cdots W_m^{\prime\prime\prime}.$$ However, then $$f(w_m)=W_m'W_m''=W_m^{\prime\prime\prime}xW_m^{\prime\prime}=W_m^{\prime\prime\prime}W^{\prime\prime}_1pW''_2\cdots W_m^{\prime\prime\prime}W_m^{\prime\prime}$$ 

contains an instance of $p$ at an index which contradicts Condition~\ref{2b}.
\subsubsection*{Case B: $|x|\ge|W_1''|,|W_m'|$} 
In this case we can write
\begin{eqnarray*}
x&=&W_1''\cdots W_j'\\
&=&W_j''\cdots W_m',
\end{eqnarray*}
for some $j$, $1<j<m$, with $W_j=W_j'W_j''$. 

 If $j>2$, then there is at least one instance of $p$ in $x=W_1''W_2\cdots W_j'$, appearing as a prefix of $W_2$. On the other hand, if $j=2$, then an instance of $p$ appears as a prefix of $W_{j+1}$ in $x=W_j''W_{j+1}\cdots W'_m$. In either case, there is at least one instance of $p$ in $x$. For the sake of definiteness, adjusting notation if necessary, choose $j$ so that $W_j'=\epsilon$ if $x$ starts with $p$; that is, assume in all cases that $W_j''\ne\epsilon$.
 
\subsubsection*{Case B(i): Word $x$ starts with $p$.} 

If $x$ starts with $p$, then Condition~\ref{2b} forces $W_1''=W_1$. Our choice of notation gives $W_j''=W_j$. Since $W_1$ and $W_j$ are prefixes of $x$, one must be a prefix of the other, and, as noted at the beginning of this proof, this forces $w_1=w_j$. Therefore $W_1=W_j$.

We prove by induction that for $1\le i\le j-2$, $w_1\cdots w_i=w_j\cdots w_{j+i-1}$, and $W_{i+1}\cdots W_{j-1}=W_{j+i}\cdots W_m'.$ We have just established the base case of this induction, when $i=1$.

Suppose that for some $k$, $1\le k< j-2$, we have $w_1\cdots w_k=w_j\cdots w_{j+k-1}$, and $W_{k+1}\cdots W_{j-1}=W_{j+k}\cdots W_m'.$ Then one of $W_{k+1}$ and $W_{j+k}$ is a prefix of the other, giving $w_{k+1}=w_{j+k}$, yielding the induction step.

Setting $i=j-1$, we see that $w_1\cdots w_{j-1}=w_j\cdots w_{2j-2}$. 
However, now $w$ contains the square $(w_1\cdots w_{j-1})^2.$ Since $|w_1\cdots w_{j-1}|\ge|p|=3$, this is a contradiction.
\subsubsection*{Case B(ii): Word $x$ doesn't start with $p$.} The first $p$ in $x$ is at the beginning of $W_2$:
If $x=W_1''W_2\cdots W_j'$ has an instance of $p$ of index $i$, $1<i< |W_1'|+1$, then $f(w_1w_2)$ contains an instance of $p$ of index properly between $1$ and $|f(w_1)|+1$, violating property 2(b). Thus the least index of $p$ in $x$ is $|W_1'|+1$. However, an analogous argument observing that $x=W_j''W_{j+1}\cdots W_m'$ yields least index of $p=|W_j''|+1.$ Thus $W_1''$ and $W_j''$ are prefixes of $x$ with the same length, forcing $W_1''=W_j''$. Now, $W_2\cdots W_j'=W_{j+1}\cdots W_m'$, so that one of $W_2$ and $W_{j+1}$ is a prefix of the other, forcing $w_2=w_{j+1}$.    

We prove by induction that for $2\le i\le j-2$, $w_2\cdots w_i=w_{j+1}\cdots w_{j+i-1}$, and $W_{i+1}\cdots W_{j-1}W_j'=W_{j+i}\cdots W_{m-1} W_m'.$ We have just established the base case of this induction, when $i=2$.

Suppose that for some $k$, $1\le k< j-1$, we have $w_2\cdots w_k=w_{j+1}\cdots w_{j+k-1}$, and $W_{k+1}\cdots W_{j-1}W_j'=W_{j+k}\cdots W_{m-1}W_m'.$ Then one of $W_{k+1}$ and $W_{j+k}$ is a prefix of the other, giving $w_{k+1}=w_{j+k}$, yielding the induction step.

When $i=j-1$, we find $W_j'=W_m'$. Since one of $W_j$ and $W_m$ must be a prefix of the other, $w_j=w_m$. Then $w$ contains the square $w_2\cdots w_j w_{j+1}\cdots w_m=(w_2\cdots w_j)^2$. Since $|w_2\cdots w_j|\ge|p|=3$, this is a contradiction.
\end{proof}

One can find morphisms on $T$ satisfying the conditions of Lemma~\ref{HN} by computer search.

It is possible to build circular FS words from circular square-free words and FS morphisms, using the following Lemma and Corollary due to Rampersad\cite{rampersad}.
\begin{lemma} If $f$ is a square-free morphism from $\Sigma_n$ to $\Sigma_m$, and $[w]$ is a square-free circular word with $|w|\ge 2$, then $[f(w)]$ is a square-free circular word. 
\end{lemma}
\begin{proof}
Write $w=w_1w_2\cdots w_\ell$, $w_\ell\in \Sigma_n$. Let $f(w_i)=W_i$, $1\le i\le \ell$. Replacing $w$ with one of its conjugates if necessary, we can assume that $W_1^{\prime\prime}W_2\cdots W_\ell W_1'$ is a representative of $[f(w)]$ containing a square, where $W_1=W_1'W_1^{\prime\prime}.$ Then $W_1W_2\cdots W_\ell W_1=f(w_1w_2\cdots w_\ell w_1)$ also contains this square. Since $f$ is square-free, this implies that $w_1w_2\cdots w_\ell w_1$ contains some square $xx$. Both $w_1w_2\cdots w_\ell$ 
and $w_2\cdots w_\ell w_1$ are representatives of $w$, and are thus square-free. It follows that $xx=w_1w_2\cdots w_\ell w_1$. However, $x$ then begins and ends with letter $w_1$, so that $w_1w_1$ appears at the center of $xx$, whence  $w$ contains the square $w_1w_1$.This is a contradiction.
\end{proof}
\begin{corollary}\label{square to circ FS} If $f$ is a FS morphism from $\Sigma_n$ to $\Sigma_2$, and $[w]$ is a square-free circular word with $|w|\ge 2$, then $[f(w)]$ is an FS circular word. 
\end{corollary}
\begin{proof} The previous proof goes through, replacing `containing a square' by `containing a square other than 00, 11, 0101', and `square-free' by `an FS morphism'. \end{proof}

To produce circular FS words with specific lengths, we make use of the following recent result by the authors \cite{curriejohnson,johnson}:

\begin{theorem}\label{ternary}
There is  a level ternary circular square-free word of length $n$,
for each positive integer $n$, $n\ne 5, 7, 9, 10, 14, 17$. \end{theorem}

Here is how we  produce circular FS words with desired lengths: Given
 $n\ge 2$, $n\ne 5, 7, 9, 10, 14, 17$, write $n=3i+j$, integers $i$ and $j$ such that $-1\le j\le 1$. Note  that $i\ge 1$. Let $w$ be a level circular square-free word over $S$ with $|w|=n$.  Permuting $a$, $b$, $c$ if necessary, assume that , $|w|_a=i+j$,  $|w|_b=|w|_c=i$. For $k\le i+j$, replacing $k$ of the $a$'s in $w$ by $d$'s gives a circular square-free word $u$ over $T$ with 
$|u|_a=i+j-k$,  $|u|_b=|u|_c=i$, $|u|_d=k$.

Suppose that $f:T^*\rightarrow B^*$ is an FS morphism,  with
$|f(a)|=\alpha$, $|f()|=\beta$, $|f(c)|=\gamma$, $|f(d)|=\delta$. By 
Corollary~\ref{square to circ FS}, $[f(u)]$ is a circular FS word, with
length 
\begin{eqnarray*}
\sum_{t\in T}|f(t)||u|_t&=&\alpha(i+j-k)+\beta i+\gamma i+\delta(k)\\
&=&(\alpha+\beta+\gamma)i+\alpha j+k(\delta-\alpha)
\end{eqnarray*}.

In a similar way, for $k\le i$, replacing $k$ of the $b$'s in $w$ by $d$'s gives a circular square-free word $v$ over $T$ with 
$|v|_a=i+j$,  $|v|_b=i-k$,$ |v|_c=i$, $|v|_d=k$,
 and  $[f(v)]$ is a circular FS word  with
length 
$$(\alpha+\beta+\gamma)i+\alpha j+k(\delta-\beta)
$$

We have proved the following:

\begin{theorem}\label{m values} Suppose there exists a FS morphism $f:T^*\rightarrow B^*$, with 
$|f(a)|=\alpha$, $|f(b)|=\beta$, $|f(c)|=\gamma$, $|f(d)|=\delta$. Then there exists a circular FS word of length $m$ for every positive integer $m$ of the form
\begin{equation}\label{length1}m=(\alpha+\beta+\gamma)i+\alpha j+k(\delta-\alpha)\end{equation}
with integers $i,j,k$ such that  
\begin{itemize}
\item $i\ge 1$
\item $-1\le j\le 1$
\item $3i+j\ne 5,7,9,10,14,17$
\item $k\le i+j$,
\end{itemize}
and
of every length $m$ of the form
\begin{equation}\label{length2}m=(\alpha+\beta+\gamma)i+\alpha j+k(\delta-\beta)\end{equation}
with  integers $i,j,k$ such that  
\begin{itemize}
\item $i\ge i$
\item $-1\le j\le 1$
\item $3i+j\ne 5,7,9,10,14,17$
\item $k\le i$.
\end{itemize}
\end{theorem}

As an example of the application of this theorem, we prove the following:

\begin{lemma}\label{7400} Suppose that $m$ is an integer,  $m\ge 7400$. There is a circular FS word of length $m$.
\end{lemma}
\begin{proof} 

One checks that the morphism $f:T^*\rightarrow S^*$ given by
\begin{eqnarray*}
f(a)&=&01100111000101110010110001011100011001011000111001\\
f(b)&=&01100111000110010110001011100101100011100101110001\\
f(c)&=&01100111000110010111000101100011100101100010111001\\
f(d)&=&011001110001011100101100011100101110001011000111001
\end{eqnarray*}
satisfies the conditions of Theorem~\ref{HN}. We have
$|f(a)|=|f(b)|=|f(c)|=50$, $|f(d)|=51.$
Write $m=50\ell+k$, $0\le k\le 49.$ Write $\ell=3i+j$, $-1\le j\le 1$. Then $\ell=\lfloor m/50\rfloor\ge 148$, so that $i=(\ell-j)/3\ge 147/3=49\ge k$. 
Then the conditions giving a length of form (\ref{length2}) hold, with $\alpha=\beta=\gamma=50$, $\delta=51$, and there is a circular FS word of length 
\begin{eqnarray*}
(\alpha+\beta+\gamma)i+\alpha j+k(\delta-\beta)&=&150i+50j+k\\
&=&150\left({\ell-j\over 3}\right)+50j+m-50\ell\\
&=&50\ell-50j+50j+m-50\ell\\
&=&m.
\end{eqnarray*}
\end{proof}

Several other morphisms satisfying the conditions of Theorem~\ref{HN} are given in Tables~\ref{m1}, \ref{m2}, and \ref{m3}.

\begin{remark}\label{knockout} Choose $r$, $0\le r\le 32$, and consider the morphism $f_r$ 
in Tables~\ref{m1}, \ref{m2}, or \ref{m3}. Let $[\alpha,\beta,\gamma,\delta]$ be a permutation of $[|f_r(a)|,|f_r(b)|,|f_r(c)|,|f_r(d)|]$. Letting $i,j,k$ take on values allowable in Theorem~\ref{m values}, one produces FS words of various lengths. A computer search thus shows that all lengths less than 7400 are obtainable in this way except for 

1, 2, 3, 4, 5, 6, 7, 8, 9, 10, 11, 12, 13, 14, 15, 16, 17, 18, 19, 20, 21, 22, 23, 24, 25, 26, 27, 28, 29, 30, 31, 32, 33, 34, 35, 36, 37, 38, 39, 40, 41, 42, 43, 44, 45, 46, 47, 48, 49, 50, 51, 52, 53, 55, 56, 57, 58, 59, 61, 63, 64, 65, 69, 70, 71, 73, 77, 116, 127, 232, 241, and 253 .
 
 A further computer search finds circular FS words in each of these cases, or shows that no such word exists. Where circular FS words exist, not obtainable via Theorem~\ref{m values} and the morphisms in Tables~\ref{m1}--\ref{m3}, they are listed in Table~\ref{other words}. The lengths for which no circular FS word exists are found to be  
 
 9, 10, 11, 13, 15, 16, 17, 18, 21, 22, 23, 25, 26, 27, 29, 31, 32, 33, 34, 35, 37, 40, 41, 42, 45, 47, 49, 53, 56, 59, 61, 64, and 73.
 \end{remark}

\begin{thm}\label{main} There is a circular FS word of length $m$ exactly when $m$ is a non-negative integer other than 9, 10, 11, 13, 15, 16, 17, 18, 21, 22, 23, 25, 26, 27, 29, 31, 32, 33, 34, 35, 37, 40, 41, 42, 45, 47, 49, 53, 56, 59, 61, 64, and 73.
\end{thm}
\begin{table}
\small
$$\begin{array}{|l|l|l|l|}
\hline
r&x&f_r(x)&|f_r(x)|\\
\hline

0&a&011001110001100101110001&24\\
&b&011001110001100101100010111001&30\\
&c&01100111000101110010110001011100011001011000111001&50\\
&d&011001110001011100101100011100101110001011000111001&51\\
\hline
1&a&110001100101110001011001&24\\
&b&11000110010110001110010110011100010111001011000101&50\\
&c&110001100101100010111001011001110001011000111001011001&54\\
&d&1100011001011000111001011100010110011100010111001011001&55\\
\hline

2&a&110001100101100010111001011001&30\\
&b&11000110010111000101100011100101110001011001&44\\
&c&11000110010110001110010110011100010111001011000101&50\\
&d&110001100101100011100101110001011000111001011000101&51\\
\hline

3&a&0110011100011001011000111001&28\\
&b&0110011100010111001011000101110001100101100010111001&52\\
&c&01100111000110010111000101100011100101110001100101100010111001&62\\
&d&011001110001011000111001011000101110001100101100011100101110001&63\\
\hline

4&a&0101100011100101100111000101110&31\\
&b&010110001110010111000101100111000110&36\\
&c&01011000111001011100011001011000101110010110011100&50\\
&d&010110001110010111000110010110001011100011001011100&51\\
\hline

5&a&010110001110010110011100011001011100&36\\
&b&01011000111001011001110001011100101100010111000110&50\\
&c&010110001110010111000110010110001011100101100111000110&54\\
&d&0101100011100101110001011001110001011100101100111000110&55\\
\hline

6&a&01011000111001011001110001011100101100010111000110&50\\
&b&0101100011100101110001011001110001011100101100111000110&55\\
&c&01011000111001011100011001011000101110010110011100011001011100&62\\
&d&010110001110010111000101100111000110010110001011100011001011100&63\\
\hline

7&a&011001110001011100101100011100101110001&39\\
&b&01100111000110010111000101100011100101110001&44\\
&c&011001110001100101100011100101110001100101100010111001&54\\
&d&0110011100011001011000101110001100101110001011000111001&55\\
\hline

8&a&0101100011100101100111000110&28\\
&b&010110001110010111000101100111000101110010110011100&51\\
&c&01011000111001011100011001011000101110010110011100011001011100&62\\
&d&010110001110010111000101100111000110010110001011100011001011100&63\\
\hline

9&a&110001100101100010111001011001&30\\
&b&110001100101100011100101110001011001&36\\
&c&1100011001011000111001011001110001011000111001011000101&55\\
&d&11000110010111000101100011100101100111000101100011100101&56\\
\hline

10&a&011001110001011100101100011100101110001&39\\
&b&01100111000101110010110001011100011001011000111001&50\\
&c&011001110001100101100011100101110001100101100010111001&54\\
&d&0110011100011001011000101110001100101110001011000111001&55\\
\hline

\end{array}
$$
\normalsize
\caption{Various FS morphisms with lengths}
\label{m1}\end{table}

\begin{table}
\small
$$\begin{array}{|l|l|l|l|}
\hline
r&x&f_r(x)&|f_r(x)|\\
\hline
11&a&110001100101100010111001011001&30\\
&b&11000110010111000101100011100101&32\\
&c&11000110010110001110010110011100010111001011000101&50\\
&d&110001100101100011100101110001011000111001011000101&51\\
\hline

12&a&11000110010111000101100011100101110001011001&44\\
&b&110001100101100011100101110001011000111001011000101&51\\
&c&110001100101100010111001011001110001011000111001011000101&57\\
&d&1100011001011100010110001110010110011100010111001011000101&58\\
\hline

13&a&000111001011001110001100101110001011&36\\
&b&000111001011000101110010110011100011001011&42\\
&c&00011100101110001100101100010111001011001110001011&50\\
&d&000111001011100010110011100010111001011001110001011&51\\
\hline

14&a&000111001011001110001011&24\\
&b&00011100101110001011001110001100101110001011&44\\
&c&0001110010111000101100111000101110010110011100011001011&55\\
&d&0001110010111000110010110001011100011001011100010110011100011001011&67\\
\hline

15&a&110011100011001011100010110001110010111000110010110001110010&60\\
&b&11001110001100101100010111001011000111001011100010110001110010&62\\
&c&11001110001100101100011100101110001011000111001011000101110010&62\\
&d&110011100011001011000101110001100101110001011000111001011100010&63\\
\hline

16&a&01100111000101110010110001011100011001011000111001&50\\
&b&01100111000110010110001011100101100011100101110001&50\\
&c&01100111000110010111000101100011100101100010111001&50\\
&d&011001110001011100101100011100101110001011000111001&51\\
\hline

17&a&110011100011001011100010&24\\
&b&110011100011001011000111001011100010110001110010&48\\
&c&110011100010111001011000111001011100011001011000101110010&57\\
&d&1100111000101110010110001011100011001011100010110001110010&58\\
\hline

18&a&110011100011001011100010&24\\
&b&110011100010111001011000111001011100010110001110010&51\\
&c&110011100011001011000111001011100011001011000101110010&54\\
&d&1100111000110010110001011100011001011100010110001110010&55\\
\hline

19&a&110011100011001011100010&24\\
&b&1100111000101110010110001011100011001011000101110010&52\\
&c&110011100011001011000111001011100011001011000101110010&54\\
&d&1100111000110010110001011100011001011100010110001110010&55\\
\hline

20&a&1100111000110010110001110010&28\\
&b&11001110001011000111001011100010&32\\
&c&110011100010111001011000111001011100010110001110010&51\\
&d&1100111000101110010110001011100011001011000101110010&52\\
\hline

21&a&1100111000110010110001110010&28\\
&b&11001110001100101110001011000111001011100010&44\\
&c&110011100010111001011000111001011100011001011000101110010&57\\
&d&1100111000101110010110001011100011001011000111001011100010&58\\
\hline

\end{array}
$$
\normalsize
\caption{Various FS morphisms with lengths}
\label{m2}\end{table}

\begin{table}
\small
$$\begin{array}{|l|l|l|l|}
\hline
r&x&f_r(x)&|f_r(x)|\\
\hline

22&a&1100111000110010110001110010&28\\
&b&1100111000101110010110001011100011001011100010&46\\
&c&110011100010111001011000111001011100010110001110010&51\\
&d&1100111000101110010110001011100011001011000101110010&52\\
\hline

23&a&1100111000110010110001110010&28\\
&b&11001110001100101110001011000111001011000101110010&50\\
&c&11001110001011000111001011100011001011000111001011100010&56\\
&d&110011100010111001011000111001011100011001011000101110010&57\\
\hline

24&a&1100111000110010110001110010&28\\
&b&11001110001100101110001011000111001011000101110010&50\\
&c&110011100010111001011000111001011100011001011000101110010&57\\
&d&1100111000101110010110001011100011001011000111001011100010&58\\
\hline

25&a&1100111000110010110001110010&28\\
&b&11001110001011000111001011100011001011000111001011100010&56\\
&c&110011100010111001011000111001011100011001011000101110010&57\\
&d&1100111000101110010110001011100011001011100010110001110010&58\\
\hline

26&a&1100111000110010110001110010&28\\
&b&110011100010111001011000111001011100011001011000101110010&57\\
&c&11001110001100101110001011000111001011100011001011000101110010&62\\
&d&110011100010110001110010110001011100011001011000111001011100010&63\\
\hline

27&a&11001110001011000111001011100010&32\\
&b&110011100011001011100010110001110010&36\\
&c&110011100011001011000111001011100011001011000101110010&54\\
&d&1100111000101110010110001110010111000110010110001110010&55\\
\hline

28&a&11001110001011000111001011100010&32\\
&b&11001110001011100101100010111000110010110001110010&50\\
&c&110011100011001011000111001011100011001011000101110010&54\\
&d&1100111000110010110001011100011001011100010110001110010&55\\
\hline

29&a&110011100011001011000111001011100010&36\\
&b&11001110001100101110001011000111001011000101110010&50\\
&c&110011100010111001011000111001011100010110001110010&51\\
&d&1100111000101110010110001011100011001011000101110010&52\\
\hline

30&a&110011100011001011100010110001110010&36\\
&b&11001110001011100101100010111000110010110001110010&50\\
&c&11001110001100101100011100101110001011000111001011000101110010&62\\
&d&110011100011001011000101110001100101110001011000111001011100010&63\\
\hline

31&a&110011100010111001011000111001011100010&39\\
&b&110011100011001011000111001011100010110001110010&48\\
&c&11001110001100101110001011000111001011100011001011000101110010&62\\
&d&110011100011001011000101110001100101110001011000111001011100010&63\\
\hline

32&a&11001110001100101110001011000111001011100010&44\\
&b&1100111000101110010110001011100011001011000101110010&52\\
&c&110011100011001011000111001011100011001011000101110010&54\\
&d&1100111000110010110001011100011001011100010110001110010&55\\
\hline

\end{array}
$$
\normalsize
\caption{Various FS morphisms with lengths}
\label{m3}
\end{table}

\begin{table}
\small
$$\begin{array}{|c|l|}
\hline
|w|&w\\
\hline
1& 0\\\hline
2& 00\\\hline
3& 000\\\hline
4& 0001\\\hline
5& 00011\\\hline
6& 000111\\\hline
7& 0001011\\\hline
8& 00010111\\\hline
12& 000101100111\\\hline
14& 00010111001011\\\hline
19& 0001011100011001011\\\hline
20& 00010110001110010111\\\hline
24& 000101100011100101100111\\\hline
28& 0001100101100011100101100111\\\hline
30& 000101110010110011100011001011\\\hline
36& 000101100011100101100111000110010111\\\hline
38& 00010110001110010110001011100101100111\\\hline
39& 000101100011100101100010111000110010111\\\hline
43& 0001011001110001011100101100111000110010111\\\hline
44& 00010110001110010111000101100111000110010111\\\hline
46& 0001011001110001011100101100010111000110010111\\\hline
48& 000101100011100101100111000110010110001110010111\\\hline
50& 00010110001110010110001011100101100111000110010111\\\hline
51& 000101100011100101100010111000110010110001110010111\\\hline
52& 0001011001110001100101100011100101100111000110010111\\\hline
55& 0001011000111001011000101110001100101100011100101100111\\\hline
57& 000101100011100101100010111000110010110001011100101100111\\\hline
58& 0001011000111001011001110001011100101100010111000110010111\\\hline
63& 000101100011100101110001011001110001011100101100111000110010111\\\hline
65& 00010110001110010110001011100011001011000101110010110001110010111\\\hline
69& 000101100011100101100010111000110010110001011100101100111000110010111\\\hline
70& 0001011000111001011000101110001100101110001011001110001011100101100111\\\hline
71& 00010110011100010111001011001110001100101100011100101100111000110010111\\\hline
77&
00010110001110010110001011100011001011100010110001110010110001011100101100111\\\hline
116& 00010110001110010110001011100011001011000101110010110001110010111000101100011\\
&100101100111000101110010110001110010111\\
\hline
127& 00011100101110001011001110001100101100010111000110010111000101100011100101100\\
&11100010111001011000101110001100101100010111001011\\\hline
232&f_{27}(abdcd)\\\hline
241&f_{29}(abdcd)\\\hline
253&f_{10}(abdcd)\\\hline
\end{array}
$$
\normalsize
\caption{Circular SF words of various lengths}\label{other words}
\end{table}

\end{document}